\documentclass[english,11pt]{article}
\usepackage[T1]{fontenc}
\usepackage[latin9]{inputenc}
\usepackage{amsthm}
\usepackage{amsmath}
\usepackage{amssymb}
\usepackage{graphicx}
\usepackage{esint}

\makeatletter

\providecommand{\tabularnewline}{\\}

\numberwithin{equation}{section}
\theoremstyle{plain}
\newtheorem{thm}{\protect\theoremname}[section]
  \theoremstyle{definition}
  \newtheorem{example}{\protect\examplename}[section]
  \theoremstyle{plain}
  \newtheorem{prop}[thm]{\protect\propositionname}
  \theoremstyle{plain}
  \newtheorem{lem}[thm]{\protect\lemmaname}
  \theoremstyle{remark}
  \newtheorem*{rem*}{\protect\remarkname}
  \theoremstyle{remark}
  \newtheorem*{acknowledgement*}{\protect\acknowledgementname}

\usepackage{url}
\usepackage[font=small]{caption}
\usepackage{fullpage}
\allowdisplaybreaks[3]

\newcommand{\optmin}{\mathrm{minimize}}
\newcommand{\optmax}{\mathrm{maximize}}
\newcommand{\optst}{\mathrm{subject\; to}}

\usepackage{authblk}

\usepackage{subfigure}
\makeatother

\usepackage{babel}
  \providecommand{\acknowledgementname}{Acknowledgement}
  \providecommand{\examplename}{Example}
  \providecommand{\lemmaname}{Lemma}
  \providecommand{\propositionname}{Proposition}
  \providecommand{\remarkname}{Remark}
\providecommand{\theoremname}{Theorem}

\begin{document}
\title{Convex Optimal Uncertainty Quantification}
\author[1]{Shuo Han\thanks{E-mail: hanshuo@seas.upenn.edu. This work was performed while the author was with the California Institute of Technology.}} 
\author[2]{Molei Tao} 
\author[1]{Ufuk Topcu} 
\author[3]{Houman Owhadi} 
\author[3]{Richard M. Murray} 
\affil[1]{Department of Electrical and Systems Engineering, University of Pennsylvania, Philadelphia, PA 19104} 
\affil[2]{School of Mathematics, Georgia Institute of Technology, Atlanta GA 30332}
\affil[3]{Division of Engineering and Applied Science, California Institute of Technology, \mbox{Pasadena, CA 91125}}
\maketitle
\begin{abstract}
Optimal uncertainty quantification (OUQ) is a framework for numerical
extreme-case analysis of stochastic systems with imperfect knowledge
of the underlying probability distribution. This paper presents sufficient
conditions under which an OUQ problem can be reformulated as a finite-dimensional
convex optimization problem, for which efficient numerical solutions
can be obtained. The sufficient conditions include that the objective
function is piecewise concave and the constraints are piecewise convex.
In particular, we show that piecewise concave objective functions
may appear in applications where the objective is defined by the optimal
value of a parameterized linear program.

\end{abstract}

\paragraph{{\small{}Key words.}}

{\small{}convex optimization, uncertainty quantification, duality
theory}{\small \par}

\paragraph{{\small{}AMS subject classifications.}}

{\small{}90C25, 90C46, 90C15, 60-08}{\small \par}

\section{Introduction\label{sec:Introduction}}

In many applications, given a cost function~$f\colon\mathbb{R}^{d}\to\mathbb{R}$
that depends on a random variable~$\theta$ in $\mathbb{R}^{d}$,
we would like to compute~$\mathbb{E}_{\theta\sim\mathcal{D}}[f(\theta)]$,
where~$\mathcal{D}$ is the probability distribution of~$\theta$.
If~$\mathcal{D}$ is known exactly, this amounts to a numerical integration
problem. However, sometimes we only have access to partial information
(e.g., moments) about~$\mathcal{D}$ that can be expressed in the
following form:
\begin{equation}
\mathbb{E}_{\theta\sim\mathcal{D}}[g(\theta)]\preceq0,\qquad\mathbb{E}_{\theta\sim\mathcal{D}}[h(\theta)]=0,\label{eq:info_constraint}
\end{equation}
where $g$ and $h$ are (real) vector-valued functions. Since the
constraints appearing in~(\ref{eq:info_constraint}) are used to
specify the available information about~$\mathcal{D}$, we will refer
to these constraints as \emph{information constraints} throughout
the paper. Using only information constraints, it is generally impossible
to compute the exact value of~$\mathbb{E}_{\theta\sim\mathcal{D}}[f(\theta)]$.
Instead, we can only hope to obtain a lower or upper bound of~$\mathbb{E}_{\theta\sim\mathcal{D}}[f(\theta)]$. 

Generally speaking, such bounds are unavailable in closed form except
for several special cases, where the bound can be obtained from probability
inequalities. To this end, this paper focuses on computing such bounds
numerically by solving infinite-dimensional optimization problems
over the set of probability distributions that satisfy the information
constraints. We follow Owhadi et al.~\cite{owhadi2010optimal} and
refer to this problem as \emph{optimal uncertainty quantification
}(OUQ) for convenience (note that the actual OUQ framework presented
in~\cite{owhadi2010optimal} is more general and is capable of dealing
with unknown functions~$f$ in addition to unknown probability distributions).
Formally, an example of OUQ problems is an optimization problem in
the following form:
\begin{alignat}{2}
 & \underset{\mathcal{D}}{\optmax} & \quad & \mathbb{E}_{\theta\sim\mathcal{D}}\left[f(\theta)\right]\label{eq:pom_ineq}\\
 & \optst & \quad & \mathbb{E}_{\theta\sim\mathcal{D}}[g(\theta)]\preceq0\label{eq:pom_ineqconstr}\\
 &  &  & \mathbb{E}_{\theta\sim\mathcal{D}}[h(\theta)]=0\label{eq:pom_eqconstr}\\
 &  &  & \theta\in\Theta\quad\text{almost surely},\label{eq:pom_support}
\end{alignat}
where~$\theta$ is a random variable in~$\mathbb{R}^{d}$, and $\mathcal{D}$
is the probability distribution of~$\theta$. The function~$f$
is real-valued. The functions $g$ and $h$ are (real) vector-valued.
The inequality~(\ref{eq:pom_ineqconstr}) is entry-wise. Formulas
(\ref{eq:pom_ineqconstr}) and (\ref{eq:pom_eqconstr}) correspond
to information constraints on~$\mathcal{D}$. For example, when~$g$
and~$h$ consist of powers of $\theta$, it implies that information
on the moments of~$\theta$ is available. The set~$\Theta\subseteq\mathbb{R}^{d}$
is the support of the distribution. For brevity, the phrase ``almost
surely'' is dropped later in this paper. Note that the condition
that $\mathcal{D}$ is a probability distribution imposes the following
implicit constraints: 
\begin{equation}
\mathbb{E}_{\theta\sim\mathcal{D}}[1]=1,\qquad\mathcal{D}\geq0.\label{eq:pom_dist_constr}
\end{equation}

\paragraph{Main result}

Despite being infinite-dimensional, a large class of OUQ problems
can be reduced to equivalent finite-dimensional optimization problems
that have the same optimal value~\cite{owhadi2010optimal}. This
reduction operates in several steps, in which the first one is a generalization
of linear programming in spaces of measures~\cite{rogosinski1958moments,shapiro2001duality}.
Although this reduction permits a numerical solution to OUQ problems,
the resulting reduced problems may be highly constrained and non-convex.

The paper focuses on the first reduction step and presents sufficient
conditions (Theorem~\ref{thm:couq_main}) under which an OUQ problem
can be reduced to a finite-dimensional \emph{convex} optimization
problem. Specifically, we require that the functions $f$, $g$, and
$h$ satisfy the following conditions:
\begin{enumerate}
\item The function~$f\colon\mathbb{R}^{d}\to\mathbb{R}$ is \emph{piecewise
concave}, i.e., it can be written as 
\begin{equation}
f(\theta)=\max_{k=1,2,\dots,K}f^{(k)}(\theta),\label{eq:pw_concave}
\end{equation}
where each function $f^{(k)}$ is concave. 
\item Each entry of the function~$g\colon\mathbb{R}^{d}\to\mathbb{R}^{p}$
is \emph{piecewise convex}, i.e., each $g_{i}$~$(i=1,2,\dots,p)$
can be written as
\begin{equation}
g_{i}(\theta)=\min_{l_{i}=1,2,\dots,L_{i}}g_{i}^{(l_{i})}(\theta),\label{eq:g_pw_convex}
\end{equation}
where each function $g_{i}^{(l_{i})}$ is convex. 
\item The function~$h\colon\mathbb{R}^{d}\to\mathbb{R}^{q}$ is affine.
Namely, it can be represented as 
\begin{equation}
h(\theta)=A^{T}\theta+b,\label{eq:h_affine}
\end{equation}
for appropriate choices of~$A\in\mathbb{R}^{d\times q}$ and~$b\in\mathbb{R}^{q}$.
\end{enumerate}

\begin{figure}
\hspace*{\fill}
\subfigure[]{\centering\includegraphics[width=0.336\textwidth]{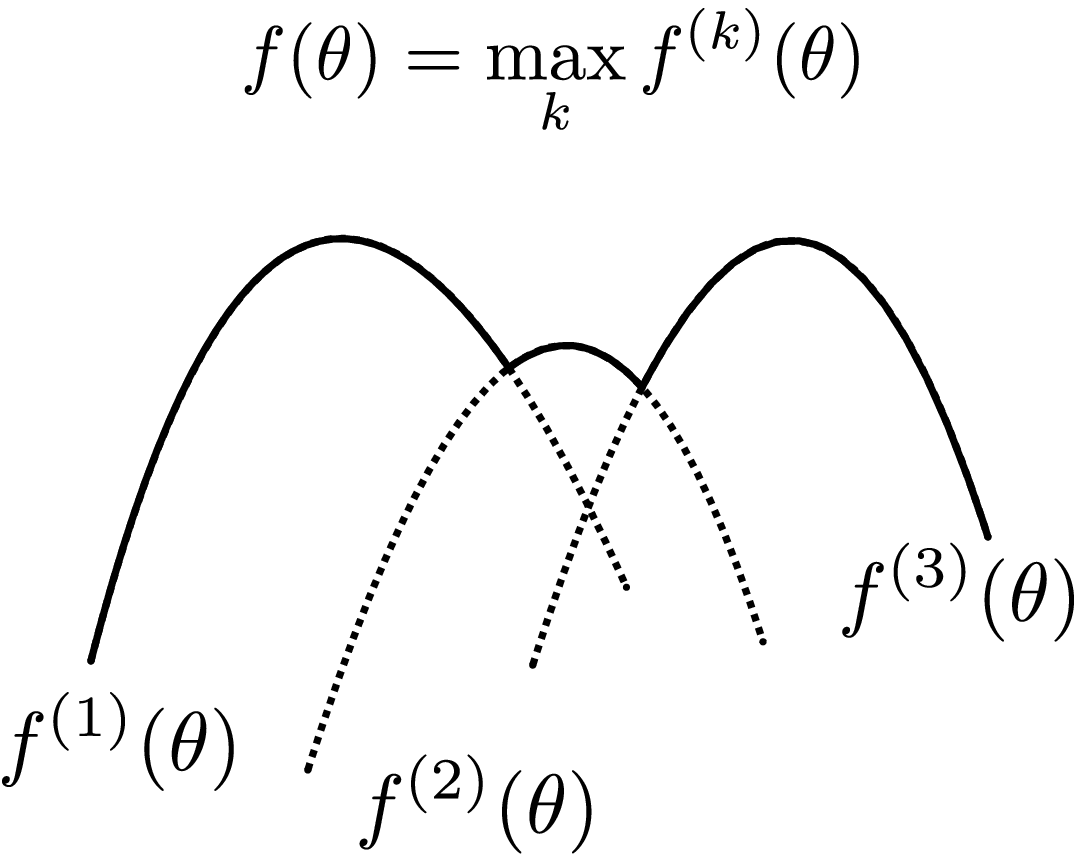}}
\hspace*{\fill}
\subfigure[]{\centering\includegraphics[bb=0bp -27bp 350bp 153bp,width=0.376\textwidth]{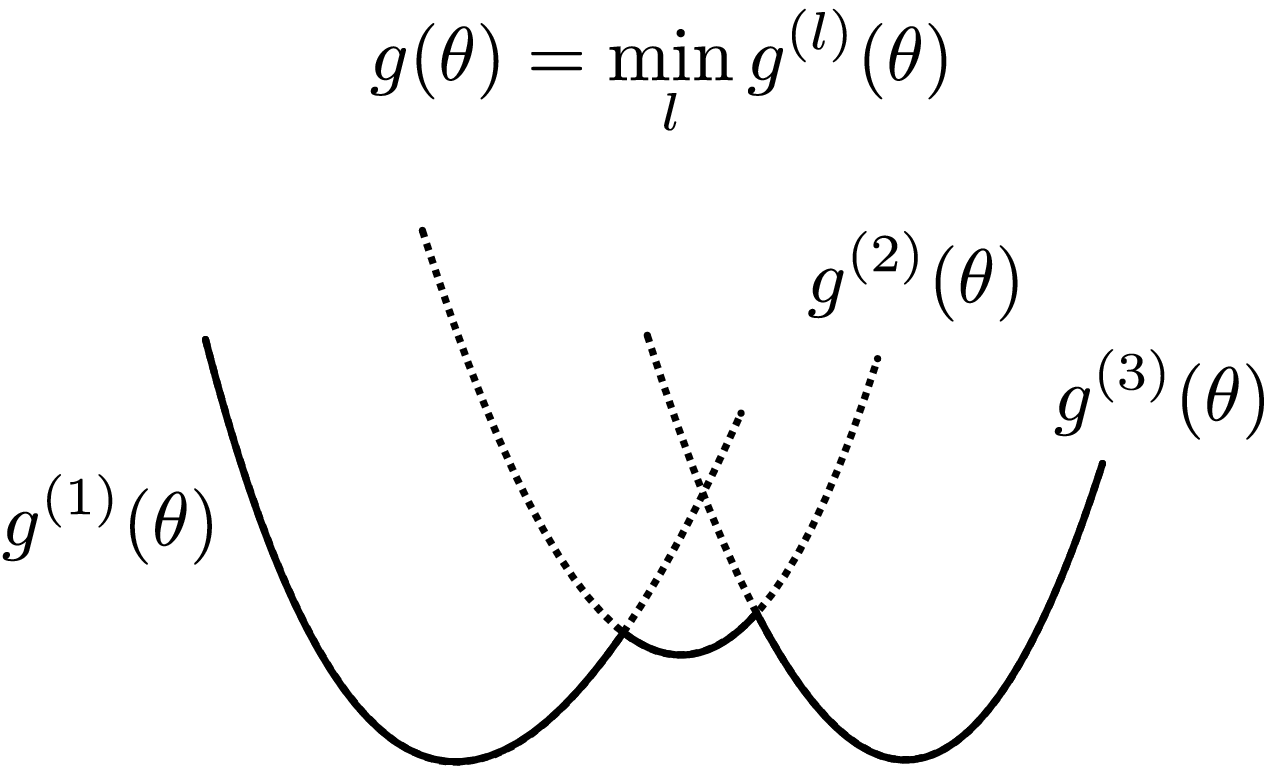}}
\hspace*{\fill}
\caption{(a) Piecewise concave and (b) piecewise convex functions in dimension one.}
\label{fig:pw_concave_convex}
\end{figure}

Figure~\ref{fig:pw_concave_convex} illustrates how piecewise concave
and piecewise convex functions look like in dimension one. In general,
these functions are neither concave nor convex. Later in section~\ref{sec:Applications},
we will give a number of useful examples of piecewise concave/convex
functions in the context of OUQ applications. 

The main result of the paper is that the OUQ problem~(\ref{eq:pom_ineq})
can be solved by considering an equivalent finite-dimensional convex
optimization problem, given that conditions~(\ref{eq:pw_concave})--(\ref{eq:h_affine})
are satisfied. For notational simplicity, we will only present the
case where $g$ is defined from~$\mathbb{R}^{d}$ to $\mathbb{R}$
(i.e., $p=1$): 
\begin{equation}
g(\theta)=\min_{l=1,2,\dots,L}g^{(l)}(\theta),\qquad g^{(l)}\mbox{ is convex.}\label{eq:g_p1}
\end{equation}
The results can be easily generalized to the case of~$p>1$. 
\begin{thm}
\label{thm:couq_main}The (convex) optimization problem 
\begin{alignat}{2}
 & \underset{\{p_{kl},\gamma_{kl}\}_{k,l}}{\optmax} & \quad & \sum_{k=1}^{K}\sum_{l=1}^{L}p_{kl}f^{(k)}(\gamma_{kl}/p_{kl})\label{eq:prob_couq_main}\\
 & \underset{\phantom{\{p_{kj},\gamma_{kj}\}_{k,j}}}{\optst} & \quad & \sum_{k=1}^{K}\sum_{l=1}^{L}p_{kl}=1\label{eq:prob_mass_constr}\\
 &  &  & p_{kl}\geq0,\quad k=1,2,\dots,K,\ l=1,2,\dots,L\label{eq:prob_mass_constr2}\\
 &  &  & \sum_{k=1}^{K}\sum_{l=1}^{L}p_{kl}h(\gamma_{kl}/p_{kl})=A^{T}\left(\sum_{k=1}^{K}\sum_{l=1}^{L}\gamma_{kl}\right)+b=0\nonumber \\
 &  &  & \sum_{k=1}^{K}\sum_{l=1}^{L}p_{kl}g^{(l)}(\gamma_{kl}/p_{kl})\leq0\label{eq:couq_pwcvx_constr}
\end{alignat}
achieves the same optimal value as problem~(\ref{eq:pom_ineq}) if
the functions~$f$, $g$, and $h$ satisfy~(\ref{eq:pw_concave}),
(\ref{eq:g_p1}), and~(\ref{eq:h_affine}), respectively.
\end{thm}

Theorem~\ref{thm:couq_main} provides sufficient conditions under
which an OUQ problem can be solved in a tractable manner, as guaranteed
by the convexity of the optimization problem~(\ref{eq:prob_couq_main}).
The proof of Theorem~\ref{thm:couq_main} will be given later in
section~\ref{sec:Convex-optimal-uncertainty}. The paper is organized
as follows. Section~\ref{sec:Optimal-uncertainty-quantificati} gives
a historical perspective of OUQ by reviewing related work and previous
results on equivalent finite-dimensional reduction of OUQ problems.
Section~\ref{sec:Convex-optimal-uncertainty} presents the main result
on convex reformulation of OUQ, which can be derived either from the
primal or the dual form of the original OUQ problem. Section~\ref{sec:Applications}
gives several applications of the convex OUQ formulation. Section~\ref{sec:Examples}
provides numerical illustrations of the main theoretical result.

\paragraph{Remarks on the OUQ formulation}

Problem~(\ref{eq:pom_ineq}) can also be written in a form without
the equality constraint~(\ref{eq:pom_eqconstr}) and the support
constraint~(\ref{eq:pom_support}). The equality constraint~(\ref{eq:pom_eqconstr})
can be eliminated by introducing the following inequality constraints:
\[
\mathbb{E}_{\theta\sim\mathcal{D}}[h(\theta)]\preceq0,\qquad\mathbb{E}_{\theta\sim\mathcal{D}}[-h(\theta)]\preceq0.
\]
The support constraint can be shown as equivalent to
\begin{equation}
\mathbb{E}_{\theta\sim\mathcal{D}}[I(\theta\notin\Theta)]\leq0,\label{eq:ind_function_1}
\end{equation}
where~$I$ is the 0-1 indicator function:
\[
I(E)=\begin{cases}
1, & \quad E=\mathrm{true}\\
0, & \quad E=\mathrm{false}.
\end{cases}
\]
In order to prove this, we note that
\[
\mathbb{E}_{\theta\sim\mathcal{D}}[I(\theta\notin\Theta)]\geq0
\]
automatically holds due to the fact that~$I$ is non-negative. Therefore,
the condition~(\ref{eq:ind_function_1}) is equivalent to
\[
\mathbb{E}_{\theta\sim\mathcal{D}}[I(\theta\notin\Theta)]=0,
\]
which is equivalent to~(\ref{eq:pom_support}). However, we will
still use the original form in problem~(\ref{eq:pom_ineq}) to distinguish
these constraints from pure inequalities. In some cases, problem~(\ref{eq:pom_ineq})
is also called the \emph{(generalized) moment problem} or \emph{problem
of moments}, since the information constraints~(\ref{eq:pom_ineqconstr})
and~(\ref{eq:pom_eqconstr}) often consist of moments of the distribution
or can be considered as generalized moments of the distribution.

\section{Optimal uncertainty quantification\label{sec:Optimal-uncertainty-quantificati}}

\subsection{Historical perspective and related work}

Among various convex formulations of OUQ, one important special case
is when
\begin{equation}
f=I(\theta\in C),\label{eq:Cset}
\end{equation}
where~$C\subseteq\mathbb{R}^{d}$ and~$I$ is the 0-1 indicator
function. Solution to the OUQ problem will yield a sharp bound of
the probability~$\mathbb{P}(\theta\in C)$ under the given information
constraints. This bound is also called Chebyshev-type bound or generalized
Chebyshev bound. The earliest theoretical analysis of such bounds
can be traced back to the pioneering work by Chebyshev and his student
Markov (see Krein~\cite{Krein} for an account of the history of
this subject along with his substantial contributions). We also refer
to early work by Isii~\cite{isii1959method,isii1962sharpness,isii1964inequalities}
and Marshall and Olkin~\cite{marshall1960multivariate}. 

As related in Owhadi and Scovel \cite[section~2]{OwhadiScovel:2013},
OUQ starts from the same mindset and applies it to more complex problems
that extend the base space to functions and measures. Instead of developing
sophisticated mathematical solutions, OUQ develops optimization problems
and reductions, so that their solution may be implemented on a computer,
as in Bertsimas and Popescu's \cite{bertsimas2005optimal} convex
optimization approach to Chebyshev inequalities, and the decision
analysis framework of Smith \cite{smith1995generalized}. Interestingly,
many inequalities in probability theory can be viewed as OUQ problems.
One such example is Markov's inequality, which has its origin in the
following problem \cite[Page~4]{Krein} (according to Krein \cite{Krein},
although Chebyshev did solve this problem, it was his student Markov
who supplied the proof in his thesis):
\begin{quote}
``Given: length, weight, position of the centroid and moment of inertia
of a material rod with a density varying from point to point. It is
required to find the most accurate limits for the weight of a certain
segment of this rod.'' 
\end{quote}
Although the statement of the problem assumes knowledge about both
the first and second moments (centroid and moment of inertia), Markov
has also obtained an inequality (known as Markov\textquoteright{}s
inequality) that is optimal with respect to the information contained
in the first moment.
\begin{example}
[Markov's inequality]Suppose~$\theta$ is a nonnegative univariate
random variable whose probability distribution is unknown, but its
mean~$\mathbb{E}[\theta]=\mu$ is given. For any $a>0$, Markov's
inequality~\cite[Page 311]{grimmett2001probability} gives a bound
for~$\mathbb{P}(\theta\geq a)$ as follows, regardless of the probability
distribution of $\theta$:
\[
\mathbb{P}(\theta\geq a)\leq\mathbb{E}[\theta]/a.
\]
Substituting $\mathbb{E}[\theta]=\mu$ into the above inequality,
we obtain
\[
\mathbb{P}(\theta\geq a)\leq\mu/a.
\]
In the OUQ framework, the problem of obtaining a tight bound for~$\mathbb{P}(\theta\geq a)$
becomes
\begin{alignat*}{2}
 & \underset{\mathcal{D}}{\optmax} & \quad & \mathbb{E}_{\theta\sim\mathcal{D}}[I(\theta\geq a)]\\
 & \optst & \quad & \mathbb{E}_{\theta\sim\mathcal{D}}[\theta]=\mu\\
 &  &  & \theta\in[0,\infty).
\end{alignat*}
In fact, it can be shown that the optimal value of the above problem
is $\mu/a$. Namely, Markov's inequality produces a tight bound.
\end{example}
Recent works on convex optimization approach to Chebyshev inequalities
(motivated by efficient numerical methods such as the interior-point
method~\cite{nesterov1994interior}) also include Lasserre~\cite{lasserre2002bounds},
Popescu~\cite{popescu2005semidefinite}, and Vandenberghe et al.~\cite{vandenberghe2007generalized}
for convex formulations for different classes of sets $C$ in (\ref{eq:Cset})
(e.g., ellipsoids, semi-algebraic sets). Compared to these formulations
that only consider the case where $f$ can be expressed in the form
of~(\ref{eq:Cset}), our formulation extends the objective function~$f$
to the more general class of piecewise concave functions as given
in~(\ref{eq:pw_concave}).

Besides indicator set functions~(\ref{eq:Cset}), another class of
functions $f$ that appear in convex formulations are functions that
are both convex and piecewise affine:
\begin{equation}
f(\theta)=\max_{k=1,2,\dots,K}\{a_{k}^{T}\theta+b_{k}\},\label{eq:piecewise_affine_convex}
\end{equation}
where~$K$ and $\{a_{k},b_{k}\}_{k=1}^{K}$ are given constants.
This form arises in applications such as stock investment~\cite{delage2010distributionally}
and logistics~\cite{bertsimas2010models}. Since all affine functions
are concave, we can define $f^{(k)}=a_{k}^{T}\theta+b_{k}$ in~(\ref{eq:pw_concave}),
so that objective functions in the form of~(\ref{eq:piecewise_affine_convex})
become a special case of our formulation. 

Besides the conditions on~$f$, there are convex formulations that
incorporate information constraints in the form of moment constraints.
Oftentimes these constraints are limited to the mean and covariance,
such as in Delage and Ye~\cite{delage2010distributionally}. For
constraints on moments of arbitrary order, it is known that the feasibility
of moment constraints can be represented as a linear matrix inequality
on a Hankel matrix consisting of the given moments~\cite[Page~170]{boyd2004convex}.
The linear matrix inequality formulation allows OUQ problems to be
cast as semidefinite programs, if the objective function is the 0-1
indicator function~\cite{popescu2005semidefinite} or a linear combination
of the moments~\cite[Page~170]{boyd2004convex}. The same technique
can also be extended to the case of complex moments and functions~\cite{byrnes2003convex}.
Compared to constraints on moments of arbitrary order, our formulation
is restricted to inequality constraints on the moments of even order
but can handle the broader class of piecewise concave objective functions.

To summarize, this paper shows that a convex reformulation of OUQ
problems can be obtained for the broader class of piecewise concave
objective functions $f$, which contains both the 0-1 indicator function~(\ref{eq:Cset})
and functions of the form~(\ref{eq:piecewise_affine_convex}) that
are considered in previous formulations by others. The information
constraints in the formulation are required to be inequalities containing
piecewise convex functions, which extend the moment constraints in
previous formulations at the expense of relaxing equality to inequality
constraints.

\subsection{Finite reduction of OUQ problems}

The perhaps surprising fact is that an OUQ problem can always be reduced
to an equivalent finite-dimensional optimization problem that yields
the same optimal value. The following proposition is originally due
to Rogosinski~\cite{rogosinski1958moments}. See also Shapiro~\cite{shapiro2001duality}
for an extension to conic linear programs and Owhadi et al.~\cite{owhadi2010optimal}
for a more general result that allows other kinds of constraints.
\begin{prop}
[Finite reduction property, cf.~\cite{owhadi2010optimal, rogosinski1958moments, shapiro2001duality}]\label{prop:finite_red_ineq}
Let~$m$ be the total number of scalar (in)equalities described by
$g$ and $h$ in (\ref{eq:pom_ineqconstr}) and (\ref{eq:pom_eqconstr}).
The finite-dimensional problem 
\begin{alignat}{2}
 & \underset{\{p_{i},\theta_{i}\}_{i=1}^{m+1}}{\optmax}\quad &  & \sum_{i=1}^{m+1}p_{i}f(\theta_{i})\label{eq:finite_red_ineq}\\
 & \underset{\phantom{\{p_{i},\theta_{i}\}_{i=1}^{m+1}}}{\optst}\quad &  & \sum_{i=1}^{m+1}p_{i}=1\nonumber \\
 &  &  & p_{i}\geq0,\quad i=1,2,\dots,m+1\nonumber \\
 &  &  & \sum_{i=1}^{m+1}p_{i}g(\theta_{i})\preceq0,\qquad\sum_{i=1}^{m+1}p_{i}h(\theta_{i})=0\nonumber \\
 &  &  & \theta_{i}\in\Theta,\quad i=1,2,\dots,m+1\nonumber 
\end{alignat}
achieves the same optimal value as problem~(\ref{eq:pom_ineq}). 
\end{prop}

Proposition~\ref{prop:finite_red_ineq} implies that the optimal
value of any OUQ problem can always be achieved by a discrete probability
distribution with finite support, whose number of Dirac masses depends
on the information constraints. In the following, we give a simple
example (reproduced from Owhadi et al.~\cite{owhadi2010optimal})
to illustrate this property.
\begin{example}
[``Seesaw'', see also~\cite{owhadi2010optimal}]\label{ex:seesaw}The
following example OUQ problem illustrates the fact that the optimal
value of an OUQ problem can be achieved by a discrete distribution
with finite support; it also shows the connection between the number
of information constraints and the number of Dirac masses required
in the (optimal) discrete distribution. Consider the following OUQ
problem for a scalar random variable~$\theta$:
\begin{alignat*}{2}
 & \underset{\mathcal{D}}{\optmax} & \quad & \mathbb{P}(\theta\geq\gamma)\\
 & \optst & \quad & \mathbb{E}_{\theta\sim\mathcal{D}}[\theta]=0,\qquad a\leq\theta\leq b,
\end{alignat*}
where~$a,b$, and~$\gamma$ are constants satisfying $a<0\leq\gamma<b.$
In order to maximize $\mathbb{P}(\theta\geq\gamma)$, we would want
to assign as much probability as possible within~$[\gamma,b]$. On
the other hand, the condition $\mathbb{E}_{\theta\sim\mathcal{D}}[\theta]=0$
requires that the probability within~$[a,0]$ and that within~$[0,b]$
must be identical. This condition is analogous to a seesaw pivoted
at $0$ with two end points at $a$ and $b$ (Figure~\ref{fig:seesaw}).
It is not difficult to see that the best assignment is to put all
the probability on the right side at $\gamma$ (for least leverage)
and all the probability on the left side at $a$ (for most leverage).
This assignment implies that the optimal distribution can be achieved
with a discrete distribution consisting of~$2$ Dirac masses at $a$
and $\gamma$. Indeed, since there is only one scalar constraint,
the total number of Dirac masses predicted by Proposition~\ref{prop:finite_red_ineq}
is~$1+1=2$.
\end{example}
\begin{figure}
\begin{centering}
\includegraphics[width=0.6\textwidth]{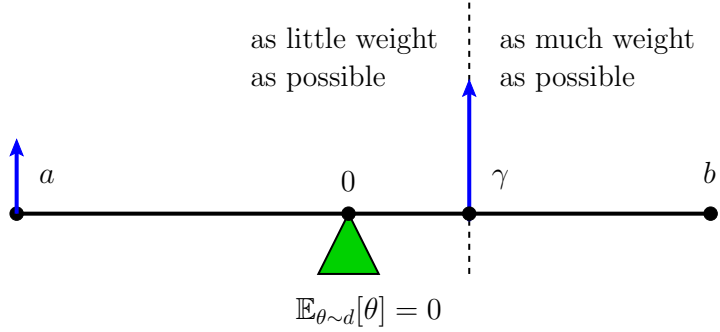}
\par\end{centering}

\caption{``Seesaw'' analogy in Example~\ref{ex:seesaw} (reproduced from~\cite{owhadi2010optimal}).}
\label{fig:seesaw}
\end{figure}

\section{Convex optimal uncertainty quantification\label{sec:Convex-optimal-uncertainty}}

We will now prove Theorem~\ref{thm:couq_main}, the main result of
this paper. Namely, under conditions~(\ref{eq:pw_concave})--(\ref{eq:h_affine}),
the OUQ problem~(\ref{eq:pom_ineq}) can be reformulated as a finite-dimensional
convex optimization problem. We will prove Theorem~\ref{thm:couq_main}
in two ways by considering the primal form and dual form of the OUQ
problem, respectively.

\subsection{Primal form\label{sec:Convex-reformulation-primal}}

According to the finite reduction property, it suffices to use a finite
number of Dirac masses to represent the optimal distribution. In particular,
due to the special form of the objective function and constraints,
these Dirac masses satisfy a useful property as given in the following
lemma.
\begin{lem}
\label{lem:max-min-achieving}Suppose $f$, $g$, and $h$ satisfy
the conditions (\ref{eq:pw_concave}), (\ref{eq:g_p1}), and (\ref{eq:h_affine}),
respectively. Then the optimal distribution for problem~(\ref{eq:pom_ineq})
can be achieved by a discrete distribution consisting of at most $K\cdot L$
Dirac masses located at $\{\theta_{kl}\}$ $(k=1,2,\dots,K;\, l=1,2,\dots,L)$.
In addition, each $\theta_{kl}$ achieves maximum at $f^{(k)}$ and
minimum at $g^{(l)}$: 
\[
f(\theta_{kl})=f^{(k)}(\theta_{kl}),\qquad g(\theta_{kl})=g^{(l)}(\theta_{kl}).
\]
\end{lem}
\begin{proof}
We will prove the lemma by contradiction. From Proposition~\ref{prop:finite_red_ineq},
we know that the optimal probability distribution can be achieved
by a discrete distribution with finite support. By way of contradiction,
assume that the optimal discrete distribution contains more than $K\cdot L$
Dirac masses. Then, by an argument using the pigeonhole principle,
the optimal discrete distribution must contain two Dirac masses located
at $\phi_{1}$ and $\phi_{2}$ with probabilities $q_{1}$ and $q_{2}$,
respectively, such that both $\phi_{1}$ and $\phi_{2}$ achieve maximum
at $f^{(k)}$ and minimum at $g^{(l)}$ for some $k\in\{1,2,\dots,K\}$
and $l\in\{1,2,\dots,L\}$, i.e., 
\begin{alignat*}{3}
f(\phi_{1}) & =f^{(k)}(\phi_{1}), & \qquad & f(\phi_{2}) & =f^{(k)}(\phi_{2}),\\
g(\phi_{1}) & =g^{(l)}(\phi_{1}), & \qquad & g(\phi_{2}) & =g^{(l)}(\phi_{2}).
\end{alignat*}
Consider a new Dirac mass whose probability~$q$ and location~$\phi$
are given by
\[
q=q_{1}+q_{2},\qquad\phi=\frac{q_{1}\phi_{1}+q_{2}\phi_{2}}{q_{1}+q_{2}}.
\]
It can be verified that replacing the two previous Dirac masses~$(q_{1},\phi_{1})$
and~$(q_{2},\phi_{2})$ with this new Dirac mass~$(q,\phi)$ will
still yield a valid probability distribution (i.e., the probability
masses sum up to 1). Moreover, the new distribution will give an objective~$\mathbb{E}[f(\theta)]$
that is no smaller than the previous one, since
\begin{equation}
qf(\phi)\geq qf^{(k)}(\phi)\geq q_{1}f^{(k)}(\phi_{1})+q_{2}f^{(k)}(\phi_{2})=q_{1}f(\phi_{1})+q_{2}f(\phi_{2}),\label{eq:obj_max_jensen}
\end{equation}
where the second inequality is an application of Jensen's inequality
and last equality uses the fact that $\phi_{1}$ and $\phi_{2}$ achieve
maximum at $f^{(k)}$. 

On the other hand, it can be shown that the new distribution will
remain as a feasible solution. The equality constraint on~$\mathbb{E}[h(\theta)]$
remains feasible, because
\begin{align*}
qh(\phi) & =q(A^{T}\phi+b)=A^{T}(q_{1}+q_{2})\phi+b(q_{1}+q_{2})\\
 & =A^{T}(q_{1}\phi_{1}+q_{2}\phi_{2})+b(q_{1}+q_{2})=q_{1}h(\phi_{1})+q_{2}h(\phi_{2}).
\end{align*}
The feasibility of the inequality constraint on~$\mathbb{E}[g(\theta)]$
can be proved by using a similar argument as in~(\ref{eq:obj_max_jensen})
by observing that $\mathbb{E}[g(\theta)]$ evaluated under the new
distribution will be no larger than that under the original distribution,
because 
\[
qg(\phi)\leq qg^{(l)}(\phi)\leq q_{1}g^{(l)}(\phi_{1})+q_{2}g^{(l)}(\phi_{2})=q_{1}g(\phi_{1})+q_{2}g(\phi_{2}).
\]
Therefore, the two old Dirac masses can be replaced by the new single
one without affecting optimality, from which the uniqueness of~$\theta_{kl}$
follows.
\qquad\end{proof}

The number of Dirac masses given by Lemma~\ref{lem:max-min-achieving}
depends only on~$K$ and~$L$, and is independent from that given
by the finite reduction property. By using Lemma~\ref{lem:max-min-achieving},
we can obtain an equivalent convex optimization problem for the original
problem~(\ref{eq:pom_ineq}), as given in Theorem~\ref{thm:couq_main}
presented in section~\ref{sec:Introduction}. The proof of Theorem~\ref{thm:couq_main}
is given as follows.
\begin{proof}
(of Theorem~\ref{thm:couq_main}) According to Lemma~\ref{lem:max-min-achieving},
we can optimize over a new set of Dirac masses whose probability weights
and locations are~$\{p_{kl},\theta_{kl}\}$ ($k=1,2,\dots,K$, $L=1,2,\dots,L$).
The requirement that the set of Dirac masses forms a valid probability
distribution imposes the constraints~(\ref{eq:prob_mass_constr})
and~(\ref{eq:prob_mass_constr2}). Under the new set of Dirac masses,
the objective function can be rewritten as
\[
\mathbb{E}\left[f(\theta)\right]=\sum_{k,l}p_{kl}f(\theta_{kl})=\sum_{k,l}p_{kl}f^{(k)}(\theta_{kl}),
\]
where the second equality uses the fact that $\theta_{kl}$ achieves
maximum at $f^{(k)}$. Unless otherwise noted, the range of the summation
$\sum_{k,l}$ over the indices $k$ and $l$ is given by $k=1,2,\dots,K$
and $l=1,2,\dots,L$. As will be shown later, this step is critical,
since $f$ is generally not concave, but $\sum_{k,l}p_{kl}f^{(k)}(\theta_{kl})$
is concave. Similarly, we have
\[
\mathbb{E}[g(\theta)]=\sum_{k,l}p_{kl}g(\theta_{kl})=\sum_{k,l}p_{kl}g^{(l)}(\theta_{kl}).
\]
The final form can be obtained by introducing new variables~$\gamma_{kl}=p_{kl}\theta_{kl}$
for all $k$ and $l$ and choosing to optimize over $\{p_{kl},\gamma_{kl}\}$
instead of $\{p_{kl},\theta_{kl}\}$. Each term in the sum in the
objective function
\[
\sum_{k,l}p_{kl}f^{(k)}(\gamma_{kl}/p_{kl})
\]
is a perspective transform of a concave function $f^{(k)}$ and hence
is concave~\cite[Page~39]{boyd2004convex}. Therefore, the objective
function is concave. Likewise, the term 
\[
\sum_{k,l}p_{kl}g^{(l)}(\gamma_{kl}/p_{kl})
\]
is convex, and corresponding constraint~(\ref{eq:couq_pwcvx_constr})
is also convex. All the other constraints are affine and do not affect
convexity. In conclusion, the final optimization problem is a finite-dimensional
convex problem and is equivalent to the original problem~(\ref{eq:pom_ineq})
due to Lemma~\ref{lem:max-min-achieving}.\qquad\end{proof}
\begin{rem*}
The perspective transformation $p_{kl}f^{(k)}(\gamma_{kl}/p_{kl})$
of $f^{(k)}$, which appears in the objective function~(\ref{eq:prob_couq_main}),
maintains the convexity of $f^{(k)}$. For certain forms of $f^{(k)}$,
the perspective transformation can be directly handled by numerical
optimization software. For example, when $f^{(k)}(x)=-x^{T}Px$ (where
$P$ is a positive semidefinite matrix) is a quadratic form, the perspective
transformation becomes $p_{kl}f^{(k)}(\gamma_{kl}/p_{kl})=-\gamma_{kl}^{T}P\gamma_{kl}/p_{kl}$,
which can be handled directly by the \texttt{quad\_over\_lin()} function
included in CVX~\cite{cvx}.
\end{rem*}

Meanwhile, there are a couple straightforward extensions to Theorem~\ref{thm:couq_main}.

\paragraph{Multiple inequality constraints}

Lemma~\ref{lem:max-min-achieving} and Theorem~\ref{thm:couq_main}
can be generalized to the case of $p>1$ based on a similar proof,
except that the number of Dirac masses becomes~$K\cdot\prod_{i=1,2,\dots,p}L_{i}$.
It can be shown that, among all Dirac masses, there is at most one
Dirac mass~$\theta^{*}$ that achieves maximum at $f^{(k)}$ and
minimum at $g_{1}^{(l_{1})},g_{2}^{(l_{2})},\dots,g_{p}^{(l_{p})}$
for any given indices~$k$ and~$\{l_{i}\}_{i=1}^{p}$:
\[
f(\theta^{*})=f^{(k)}(\theta^{*}),\quad g_{1}(\theta^{*})=g_{1}^{(l_{1})}(\theta^{*}),\quad\dots,\quad g_{p}(\theta^{*})=g_{p}^{(l_{p})}(\theta^{*}).
\]
The corresponding convex optimization problem can be formed by following
a similar procedure as in the proof of Theorem~\ref{thm:couq_main}.

\paragraph{Support constraints}

It is possible to impose certain types of constraints on the support~$\Theta$
without affecting convexity. Specifically, we can allow
\begin{equation}
\Theta=\bigcup_{s=1}^{S}C^{(s)},\label{eq:union_convex}
\end{equation}
where each $C^{(s)}\subseteq\mathbb{R}^{d}$ is a convex set. In order
to show that the corresponding OUQ problem remains convex, we use
the fact that the support constraint~(\ref{eq:pom_support}) is equivalent
to
\begin{equation}
\mathbb{E}[I(\theta\notin\Theta)]\leq0,\label{eq:union_convex_ineq}
\end{equation}
as presented at the beginning of section~\ref{sec:Optimal-uncertainty-quantificati}.
When~$\Theta$ satisfies~(\ref{eq:union_convex}), we have
\begin{equation}
I(\theta\notin\Theta)=\min\{1,I_{\infty}^{0}(\theta\in C^{(1)}),\dots,I_{\infty}^{0}(\theta\in C^{(S)}\},\label{eq:union_convex_ind_func}
\end{equation}
which is piecewise convex. 

If the support constraint~$\theta\in\Theta$ (or equivalently~(\ref{eq:union_convex_ineq}))
is added to an OUQ problem where~$f$ and~$g$ satisfy~(\ref{eq:pw_concave})
and~(\ref{eq:g_p1}), the number of Dirac masses becomes~$K\cdot L\cdot(S+1)$.
Denote these Dirac masses as~$\{p_{kl},\gamma_{kl}\}_{k,l}$ and
$\{p_{kls},\gamma_{kls}\}_{k,l,s}$ ($k=1,2,\dots,K$; $l=1,2,\dots,L$;
$s=1,2,\dots,S$). Then inequality constraint~(\ref{eq:union_convex_ineq})
becomes
\[
\sum_{k,l}\left[p_{kl}+\sum_{s=1}^{S}p_{kls}I_{\infty}^{0}(\gamma_{kls}/p_{kls}\in C^{(s)})\right]\leq0,
\]
where the first term appears due to the constant~$1$ in~(\ref{eq:union_convex_ind_func}).
Recall that $p_{kl}\geq0$ for all $k$ and~$l$. Then we have
\begin{equation}
p_{kl}=0,\qquad\gamma_{kls}/p_{kls}\in C^{(s)},\quad s=1,2,\dots,S\label{eq:union_convex_constr}
\end{equation}
for all~$k$ and~$l$. The fact~$p_{kl}=0$ implies that the actual
number of Dirac masses needed in this case is~$K\cdot L\cdot S$. 

In particular, when each~$C^{(s)}$ is a convex polytope, the corresponding
support constraint has a simpler form. It is known that any convex
polytope can be represented as an intersection of affine halfspaces.
Namely, each~$C^{(s)}$ can be expressed as: 
\[
C^{(s)}=\{\theta\colon A^{(s)}\theta\preceq b^{(s)}\}
\]
for certain constants~$A^{(s)}$ and $b^{(s)}$~(cf.~\cite{bertsimas1997introduction}).
Then the support constraint~(\ref{eq:union_convex_constr}) becomes
affine constraints 
\[
A^{(s)}\gamma_{kls}\preceq p_{kls}b^{(s)},\quad s=1,2,\dots,S
\]
for all~$k$ and~$l$.

\subsection{Dual form\label{sec:Convex-reformulation-dual}}

The convex reformulation of OUQ that appears in Theorem~\ref{thm:couq_main}
can also be derived from the Lagrange dual problem of~(\ref{eq:pom_ineq}).
Derivation from the dual problem is particularly useful in \emph{distributionally
robust stochastic programming}, as studied by Delage and Ye~\cite{delage2010distributionally},
Mehrotra and Zhang~\cite{mehrotra2013models}. First of all, the
Lagrangian of problem~(\ref{eq:pom_ineq}) can be written as 
\begin{align*}
L & =\int f(\theta)\mathcal{D}(\theta)\, d\theta-\lambda^{T}\int g(\theta)\mathcal{D}(\theta)\, d\theta-\nu^{T}\int h(\theta)\mathcal{D}(\theta)\, d\theta\\
 & \qquad+\int\lambda_{p}(\theta)\mathcal{D}(\theta)\, d\theta+\mu\left(1-\int\mathcal{D}(\theta)\, d\theta\right).
\end{align*}
The last two terms are due to the constraints~(\ref{eq:pom_dist_constr}).
From the Lagrangian, the Lagrange dual can be derived as
\[
\sup_{\mathcal{D}}L=\begin{cases}
\mu, & f(\theta)-\lambda^{T}g(\theta)-\nu^{T}h(\theta)+\lambda_{p}(\theta)-\mu=0\textrm{ for all }\theta\\
\infty, & \textrm{otherwise}.
\end{cases}
\]
By including the conditions on the Lagrange multipliers, i.e.,
\[
\lambda\succeq0,\qquad\lambda_{p}(\theta)\geq0,\quad\forall\theta,
\]
we can obtain the the dual problem as
\begin{alignat}{2}
 & \underset{\lambda,\nu,\mu}{\optmin} & \quad & \mu\label{eq:pom_ineq_dual}\\
 & \underset{\phantom{\lambda,\nu,\mu}}{\optst} & \quad & f(\theta)-\lambda^{T}g(\theta)-\nu^{T}h(\theta)-\mu\leq0,\quad\forall\theta\label{eq:pom_ineq_dual_constr}\\
 &  &  & \lambda\succeq0,\nonumber 
\end{alignat}
which is a linear program with an infinite number of constraints (also
known as a semi-infinite linear program). Under standard constraint
qualifications, as shown by Isii~\cite{isii1959method,isii1962sharpness,isii1964inequalities},
strong duality holds so that we can solve the dual problem~(\ref{eq:pom_ineq_dual})
to obtain the optimal value of problem~(\ref{eq:pom_ineq}). Analysis
on strong duality can also be found in Karlin and Studden~\cite{karlin1966tchebycheff},
Akhiezer and Krein~\cite{akhiezer1962some}, Smith~\cite{smith1995generalized},
and Shapiro~\cite{shapiro2001duality}.

The inequality constraint~(\ref{eq:pom_ineq_dual_constr}) implies
that the optimal solution~$(\lambda^{*},\nu^{*},\mu^{*})$ must satisfy
\[
\mu^{*}=\max_{\theta}\left[f(\theta)-\lambda^{*T}g(\theta)-\nu^{*T}h(\theta)\right],
\]
which allows us to eliminate the inequality constraint~(\ref{eq:pom_ineq_dual_constr})
and rewrite problem~(\ref{eq:pom_ineq_dual}) as
\begin{alignat}{2}
 & \underset{\lambda,\nu}{\optmin} & \quad & \max_{\theta}\left[f(\theta)-\lambda^{T}g(\theta)-\nu^{T}h(\theta)\right]\label{eq:pom_ineq_dual-2}\\
 & \underset{\phantom{\lambda,\nu}}{\optst} & \quad & \lambda\succeq0.\nonumber 
\end{alignat}
As it turns out, Theorem~\ref{thm:couq_main} can also be proved
from the dual form~(\ref{eq:pom_ineq_dual-2}). Similar to section~\ref{sec:Convex-reformulation-primal},
we will only prove for the case of~$p=1$ for notational convenience.
First, we present the following lemma for later use in the proof.
\begin{lem}
\label{lem:discrete_to_prob_simplex}Given a set of real-valued functions~$\{f^{(k)}\}_{k=1}^{K}$,
the optimal value of the optimization problem
\begin{alignat}{2}
 & \underset{\{p_{k},\theta_{k}\}_{k=1}^{K}}{\optmax} & \quad & \sum_{k=1}^{K}p_{k}f^{(k)}(\theta_{k})\label{eq:convex_comb_optimization-1}\\
 & \underset{\phantom{\{p_{k},\theta_{k}\}_{k=1}^{K}}}{\optst} & \quad & \sum_{k=1}^{K}p_{k}=1,\qquad p_{k}\geq0,\quad k=1,2,\dots,K\nonumber 
\end{alignat}
 is $\max_{\theta}\max_{k=1,2,\dots,K}\{f^{(k)}(\theta)\}$.\end{lem}
\begin{proof}
Denote the optimal value of problem~(\ref{eq:convex_comb_optimization-1})
as~$\mathrm{OPT}$ and
\[
\theta^{*}=\arg\max_{\theta}\max_{k=1,2,\dots,K}\{f^{(k)}(\theta)\},\qquad k^{*}=\arg\max_{k=1,2,\dots,K}\{f^{(k)}(\theta^{*})\}.
\]
Then we have $\mathrm{OPT}\geq f^{(k^{*})}(\theta^{*})=\max_{\theta}\max_{k=1,2,\dots,K}\{f^{(k)}(\theta)\}$,
since 
\[
p_{k}=\begin{cases}
1, & \quad k=k^{*}\\
0, & \quad\textrm{otherwise},
\end{cases}\qquad\theta_{k}=\theta^{*},\quad\forall k
\]
is a feasible solution of problem~(\ref{eq:convex_comb_optimization-1}),
and its corresponding objective value is~$f^{(k^{*})}(\theta^{*})$.
On the other hand, suppose~$\{p_{k}^{*},\theta_{k}^{*}\}_{k=1}^{K}$
is the optimal solution of problem~(\ref{eq:convex_comb_optimization-1}).
Then we have
\begin{align*}
\mathrm{OPT} & =\sum_{k=1}^{K}p_{k}^{*}f^{(k)}(\theta_{k}^{*})\leq\sum_{k=1}^{K}\left[p_{k}^{*}\max_{k=1,2,\dots,K}\left\{ f^{(k)}(\theta_{k}^{*})\right\} \right]\\
 & =\left(\sum_{k=1}^{K}p_{k}^{*}\right)\cdot\max_{k=1,2,\dots,K}\left\{ f^{(k)}(\theta_{k}^{*})\right\} =\max_{k=1,2,\dots,K}\left\{ f^{(k)}(\theta_{k}^{*})\right\} \\
 & \leq\max_{k=1,2,\dots,K}\left\{ \max_{\theta}f^{(k)}(\theta)\right\} =\max_{\theta}\max_{k=1,2,\dots,K}\left\{ f^{(k)}(\theta)\right\} .
\end{align*}
Therefore, we have~$\mathrm{OPT}=\max_{\theta}\max_{k=1,2,\dots,K}\{f^{(k)}(\theta)\}$.
\qquad\end{proof}

We are now ready to prove Theorem~\ref{thm:couq_main} from the dual
problem~(\ref{eq:pom_ineq_dual-2}).
\begin{proof}
(of Theorem~\ref{thm:couq_main}) For convenience, we define the
objective function in~(\ref{eq:pom_ineq_dual-2}) as
\[
L(\lambda,\nu)=\max_{\theta}\left[f(\theta)-\lambda g(\theta)-\nu^{T}h(\theta)\right],
\]
where~$\lambda$ is now reduced to a scalar in the case of~$p=1$.
Recall that 
\[
f(\theta)=\max_{k=1,2,\dots,K}f^{(k)}(\theta),\qquad g(\theta)=\min_{l=1,2,\dots,L}g^{(l)}(\theta).
\]
Because~$\lambda\geq0$, we have 
\[
L(\lambda,\nu)=\max_{\theta}\max_{k,l}\left\{ f^{(k)}(\theta)-\lambda g^{(l)}(\theta)-\nu^{T}h(\theta)\right\} 
\]
and, by Lemma~\ref{lem:discrete_to_prob_simplex}, 
\[
L(\lambda,\nu)=\max_{\{p_{kl},\theta_{kl}\}_{k,l}}\sum_{k,l}p_{kl}\left[f^{(k)}(\theta_{kl})-\lambda g^{(l)}(\theta_{kl})-\nu^{T}h(\theta_{kl})\right],
\]
where $\{p_{kl}\}$ ($k=1,2,\dots,K$; $l=1,2,\dots,L$) need to satisfy~$\sum_{k,l}p_{kl}=1$
and~$p_{kl}\geq0$ for all $k$ and $l$. Similar to the previous
proof in section~\ref{sec:Convex-reformulation-primal}, we introduce
new variables~$\gamma_{kl}=p_{kl}\theta_{kl}$, and rewrite~$L(\lambda,\nu)$
as 
\[
L(\lambda,\nu)=\max_{\{p_{kl},\gamma_{kl}\}_{k,l}}\sum_{k,l}\left[p_{kl}f^{(k)}(\gamma_{kl}/p_{kl})-\lambda p_{kl}g^{(l)}(\gamma_{kl}/p_{kl})-\nu^{T}p_{kl}h(\gamma_{kl}/p_{kl})\right].
\]
Next, because $f^{(k)}$ is concave and~$g^{(l)}$ is convex for
all $k$ and $l$, and $h$ is affine, if problem~(\ref{eq:pom_ineq_dual-2})
is feasible, then the optimal solution is a saddle point of
\[
\sum_{k,l}\left[p_{kl}f^{(k)}(\gamma_{kl}/p_{kl})-\lambda p_{kl}g^{(l)}(\gamma_{kl}/p_{kl})-\nu^{T}p_{kl}h(\gamma_{kl}/p_{kl})\right].
\]
Therefore, problem~(\ref{eq:pom_ineq_dual-2}) achieves the same
optimal value as the following problem, obtained by exchanging the
order of maximization and minimization:
\begin{alignat}{2}
 & \underset{\{p_{kl},\gamma_{kl}\}}{\optmax} & \quad & \min_{\lambda\geq0,\nu}\sum_{k,l}\left[p_{kl}f^{(k)}(\gamma_{kl}/p_{kl})-\lambda p_{kl}g^{(l)}(\gamma_{kl}/p_{kl})-\nu^{T}p_{kl}h(\gamma_{kl}/p_{kl})\right]\label{eq:pom_ineq_dual-2-1}\\
 & \underset{\phantom{\lambda,\nu,\mu}}{\optst} & \quad & \sum_{k,l}p_{kl}=1,\qquad p_{kl}\geq0,\quad\forall k,l.\nonumber 
\end{alignat}
Using the fact
\begin{align*}
 & \min_{\lambda\geq0,\nu}\sum_{k,l}\left[p_{kl}f^{(k)}(\gamma_{kl}/p_{kl})-\lambda p_{kl}g^{(l)}(\gamma_{kl}/p_{kl})-\nu^{T}p_{kl}h(\gamma_{kl}/p_{kl})\right]\\
 & \qquad=\begin{cases}
\sum_{k,l}p_{kl}f^{(k)}(\gamma_{kl}/p_{kl}), & \quad\sum_{k,l}p_{kl}g^{(l)}(\gamma_{kl}/p_{kl})\leq0,\textrm{ }\sum_{k,l}p_{kl}h(\gamma_{kl}/p_{kl})=0\\
-\infty, & \quad\textrm{otherwise},
\end{cases}
\end{align*}
we can further rewrite problem~(\ref{eq:pom_ineq_dual-2-1}) as problem~(\ref{eq:prob_couq_main})
in Theorem~\ref{thm:couq_main}.
\qquad\end{proof}
\begin{rem*}
The dual form~(\ref{eq:pom_ineq_dual-2}) is often used to derive
a numerically favorable solution for distributionally robust stochastic
programming, in which an optimization problem in the following form
is being solved:
\begin{equation}
\underset{x}{\optmin}\quad\max_{\mathcal{D}\in\Delta_{\mathcal{D}}}\mathbb{E}_{\theta\sim\mathcal{D}}\left[f(x,\theta)\right],\label{eq:drsp}
\end{equation}
where the set $\Delta_{\mathcal{D}}$ is defined by the information
constraints (\ref{eq:pom_ineqconstr})--(\ref{eq:pom_support}). Using
the dual form (\ref{eq:pom_ineq_dual-2}), it can be seen that the
distributionally robust stochastic programming problem (\ref{eq:drsp})
can be rewritten as
\begin{alignat}{2}
 & \underset{x,\lambda,\nu}{\optmin} & \quad & \max_{\theta}\left[f(x,\theta)-\lambda^{T}g(\theta)-\nu^{T}h(\theta)\right]\label{eq:drsp-2}\\
 & \underset{\phantom{x,\lambda,\nu}}{\optst} & \quad & \lambda\succeq0.\nonumber 
\end{alignat}
If $f$ is convex in $x$, then the objective function in (\ref{eq:drsp-2})
is a pointwise maximum of convex functions of $x$, $\lambda$, and
$\nu$, and hence problem (\ref{eq:drsp-2}) is convex. In some cases
(e.g., Delage and Ye~\cite{delage2010distributionally}, Mehrotra
and Zhang~\cite{mehrotra2013models}), $ $the objective function
in (\ref{eq:drsp-2}) can be rewritten using linear matrix inequalities
to permit a more efficient numerical solution.\end{rem*}

\section{Applications\label{sec:Applications}}

In this section, we illustrate the applications of the convex OUQ
framework through a number of examples. In particular, we show that
piecewise concave functions are expressive enough in modeling the
objective function for several different applications.

\subsection{Piecewise convex functions as information constraints}

In the next, we list several examples of piecewise convex functions
that can be used as information constraints.
\begin{example}
[Even-order moments]Consider the case where the random variable~$\theta$
is univariate and the function $g$ in the inequality constraint is
given by $g(\theta)=\theta^{2q}$ (for some positive integer $q$).
It can be verified that $g$ is convex, which is a special case of
piecewise convex functions.
\end{example}
\begin{example}
\label{ex:pwcvx_tail}Consider~$g(\theta)=I(\theta\notin C)$ for
any convex set~$C\in\mathbb{R}^{d}$. The function $g$ can be used
to specify the probability~$\mathbb{P}(\theta\notin C)$, since~$\mathbb{P}(\theta\notin C)=\mathbb{E}[I(\theta\notin C)]$.
The function~$g$ can be rewritten as
\[
g(\theta)=\min\{1,I_{\infty}^{0}(\theta\in C)\},
\]
where the function~$I_{b}^{a}$ is defined as
\begin{equation}
I_{b}^{a}(E)=\begin{cases}
a, & \quad E=\mathrm{true}\\
b, & \quad E=\mathrm{false}
\end{cases}\label{eq:defn_indicator}
\end{equation}
for any~$a,b\in\mathbb{R}$. It can be verified that both $1$ and
$I_{\infty}^{0}(\theta\in C)$ are convex in~$\theta$, and hence
$g$ is piecewise convex.
\end{example}

\subsection{Piecewise concave functions as objectives}

We begin by two simple examples of piecewise concave functions that
are used as objective functions in OUQ.
\begin{example}
[Convex and piecewise affine]When $f^{(k)}$ is affine (hence concave)
for each $k\in\{1,2,\dots,K\}$, the function~$f=\max_{k=1,2,\dots,K}\{f^{(k)}\}$
becomes convex and piecewise affine. This class of functions has been
studied by, e.g., Delage and Ye~\cite{delage2010distributionally}
and Bertsimas et al.~\cite{bertsimas2010models}.
\end{example}
\begin{example}
\label{ex:worst_case_prob_dis_pwc}Consider~$f(\theta)=I(\theta\in C)$
for any convex set~$C\subseteq\mathbb{R}^{d}$. The function $f$
can be used to specify the probability~$\mathbb{P}(\theta\in C)$,
since~$\mathbb{P}(\theta\in C)=\mathbb{E}[I(\theta\in C)]$. The
function $f$ can be rewritten as 
\[
f(\theta)=\max\{0,I_{-\infty}^{1}(\theta\in C)\},
\]
where the the definition of $I_{-\infty}^{1}$ follows from~(\ref{eq:defn_indicator}).
It can be verified that both~$0$ and~$I_{-\infty}^{1}(\theta\in C)$
are concave in~$\theta$, and hence $f$ is piecewise concave.
\end{example}
Aside from the two simple examples presented above, we present in
the following a very important form of objective functions that is
also piecewise concave. 
\begin{example}
[Optimal value of parameterized linear programs]This form of objective
functions is defined by the optimal value of a parameterized linear
program whose constraints contain nonlinear terms of the random variable~$\theta$.
The following is the linear program under consideration:
\begin{alignat}{2}
 & \underset{x}{\optmin} & \quad & c^{T}x\label{eq:prob_opt}\\
 & \underset{\phantom{x}}{\optst} & \quad & Ax\preceq u(\theta),\qquad Hx=v,\nonumber 
\end{alignat}
where $x,c\in\mathbb{R}^{n}$, $v\in\mathbb{R}^{q}$, $A\in\mathbb{R}^{m\times n}$,
$H\in\mathbb{R}^{q\times n}$, and each component of the function
$u\colon\mathbb{R}^{p}\to\mathbb{R}^{m}$ is convex in $\theta$. 

For any fixed $\theta$, denote the optimal value of problem~(\ref{eq:prob_opt})
by $f(\theta)$. We assume that there exists a nonempty set $\Theta\subseteq\mathbb{R}^{p}$
such that problem~(\ref{eq:prob_opt}) is feasible for all $\theta\in\Theta$,
i.e., $f(\theta)<\infty$ for all $\theta\in\Theta$. We will show
that $f$ is piecewise concave on $\Theta$. Namely, $f$ can be expressed
as
\[
f(\theta)=\max_{k\in\{1,2,\dots,K\}}f^{(k)}(\theta)
\]
for some $K$ and concave functions $\{f^{(k)}\}_{k=1}^{K}$. First,
we consider the Lagrange dual problem of problem~(\ref{eq:prob_opt}),
which is given as follows:
\begin{alignat}{2}
 & \underset{\lambda,\mu}{\optmax} & \quad & -\lambda^{T}u(\theta)-\mu^{T}v\label{eq:prob_opt-1}\\
 & \underset{\phantom{\lambda,\mu}}{\optst} & \quad & A^{T}\lambda+H^{T}\mu+c=0,\qquad\lambda\succeq0.\nonumber 
\end{alignat}
Since problem~(\ref{eq:prob_opt}) is feasible, strong duality between
problems~(\ref{eq:prob_opt}) and~(\ref{eq:prob_opt-1}) holds,
so that $f(\theta)$ is also the optimal value of problem~(\ref{eq:prob_opt-1}).
Notice that problem~(\ref{eq:prob_opt-1}) is a linear program for
any given $\theta$. Since the optimal value of a linear program can
always be achieved at a vertex of the constraint set, we can rewrite
$f$ as 
\[
f(\theta)=\max_{k\in\{1,2,\dots,K\}}\{-\lambda_{k}^{T}u(\theta)-\mu_{k}^{T}v\},
\]
where $\{(\lambda_{k},\mu_{k})\}_{k=1}^{K}$ is the set of vertices
of the polytope 
\[
\{(\lambda,\mu)\colon A^{T}\lambda+H^{T}\mu+c=0,\ \lambda\succeq0\}.
\]
Since $\lambda_{k}\succeq0$ for all $k\in\{1,2,\dots,K\}$, and each
component of $u$ is convex, we know that $-\lambda_{k}^{T}u-\mu_{k}^{T}v$
is concave, which implies that $f$ is piecewise concave.
\end{example}
In the following, we give two applications in which the objective
function is defined through the optimal value of a linear program.

\paragraph{Revenue maximization with stochastic supplies}

We start with a simple application that cannot be handled by previous
convex formulations of OUQ. Consider a scenario where a merchant would
like to estimate the expected revenue from selling (divisible) goods
to $K$ potential customers, whereas the quantity of the goods follows
an unknown probability distribution. We assume that the merchant can
only choose to sell the goods exclusively to one of the $K$ customers.
Each customer offers a different price, and the merchant tries to
maximize the revenue by selling to the one who offers the best price.
In addition, price from each customer drops as the quantity increases,
which makes the revenue a nonlinear function of the quantity. This
model can capture the situation in which the customer gradually loses
interest in purchasing as the quantity increases. Eventually, a customer
stops purchasing beyond a certain maximum quantity, at which point
the merchant can no longer increase the revenue by selling to that
customer. 

We denote by $\theta$ the total quantity of goods. We denote by $f^{(k)}(\theta)\in\mathbb{R}$
the corresponding revenue of selling $\theta$ quantity of goods to
customer $k$. In the following, we shall assume that $f^{(k)}$ is
concave. Then, the maximum revenue of selling $\theta$ quantity of
goods is given by 
\[
f(\theta)=\max_{k=1,2,\dots,K}f^{(k)}(\theta),
\]
which is piecewise concave by definition. We can also write $f$ (as
a function of $\theta$) as the optimal value of the following optimization
problem:
\begin{alignat}{2}
 & \underset{\mu}{\optmin} & \quad & \mu\label{eq:prob-revenue}\\
 & \underset{}{\optst} & \quad & \mu\geq f^{(k)}(\theta),\quad k=1,2,\dots,K.\nonumber 
\end{alignat}
It can be verified that problem (\ref{eq:prob-revenue}) has the same
form as problem~(\ref{eq:prob_opt}), which also implies that $f$
is piecewise concave.

\paragraph{DC optimal power flow with stochastic demands}

Consider a power network modeled as a graph $\mathcal{G}=(\mathcal{N},\mathcal{E})$,
where $\mathcal{N}$ is the set of nodes, and $\mathcal{E}$ is the
set of edges. We assume that a subset $\mathcal{S}\subseteq\mathcal{N}$
of the nodes are capable of generating electric power to supply other
loads in the network. The power generation at any node $i\in\mathcal{S}$
is denoted by $s_{i}\in\mathbb{R}$. For convenience, we assume $s_{i}=0$
for all $i\in\mathcal{N}\backslash\mathcal{S}$. We assume that the
power demand at any node $i\in\mathcal{N}$ is stochastic, and the
demand can be modeled by a function $d_{i}(\theta)$ that depends
on some random variable $\theta$. As an example, the power demand
that comes from the usage of air conditioning depends on the ambient
temperature, which can be modeled as a random variable. Specifically,
we can define $\theta\in\mathbb{R}^{|\mathcal{N}|}$ as the vector
of ambient temperatures at different nodes (corresponding to different
geographical regions), so that $d_{i}$ depends on the temperature
$\theta_{i}$ at node $i$. We shall assume that the demand function
$d_{i}$ is concave (in fact, also monotonic) in $\theta_{i}$ for
all $i\in\mathcal{N}$.

We consider a simplified DC power flow model of the network adopted
from Stott et al.~\cite{stott2009dc}. Denote by $q_{ij}\in\mathbb{R}$
the power flow from node $i\in\mathcal{N}$ to node $j\in\mathcal{N}$.
The power flow is determined by the difference in the voltage angles
$\alpha_{i},\alpha_{j}\in\mathbb{R}$ of $i$ and $j$ as follows:
\begin{equation}
q_{ij}=B_{ij}(\alpha_{i}-\alpha_{j}),\label{eq:opf-qij}
\end{equation}
where $B_{ij}=B_{ji}$ is the susceptance of the transmission line
between $i$ and $j$. The amount of power flow is also limited by
the transmission capacity $\bar{q}_{ij}\in\mathbb{R}_{+}$ of the
line connecting nodes $i$ and $j$:
\begin{equation}
-\bar{q}_{ij}\leq q_{ij}\leq\bar{q}_{ij}.\label{eq:opf-qij_bound}
\end{equation}
The net supply at node $i\in\mathcal{N}$ is given by
\begin{equation}
p_{i}=\sum_{j:(j,i)\in\mathcal{E}}q_{ji}+s_{i}.\label{eq:opf-pi}
\end{equation}
We require that there should be enough net supply to satisfy the demand
at any node $i\in\mathcal{N}$, which implies
\begin{equation}
d_{i}(\theta)\leq p_{i},\quad i\in\mathcal{N}.\label{eq:opf-di}
\end{equation}
The goal of the DC optimal power flow problem is to minimize the total
generation cost while respecting the constraints (\ref{eq:opf-qij})--(\ref{eq:opf-di}).
We assume that the generation cost $c_{i}$ at node $i$ can be modeled
by a convex and piecewise affine function:
\[
c_{i}(s_{i})=\max_{r=1,2,\dots,R_{i}}\{a_{ir}^{T}s_{i}+b_{ir}\},
\]
where $R_{i}$ and $\{a_{ir},b_{ir}\}_{r=1}^{R_{i}}$ are given constants.
Then, the DC optimal power flow problem can be cast as an optimization
problem as follows:
\begin{alignat*}{2}
 & \underset{}{\optmin} & \quad & \sum_{i\in\mathcal{S}}c_{i}(s_{i})\qquad\text{(over \ensuremath{s_{i},q_{ij},\alpha_{i},p_{i}})}\\
 & \underset{}{\optst} & \quad & \text{\eqref{eq:opf-qij}--\eqref{eq:opf-di}.}
\end{alignat*}
We can rewrite the above optimization problem as a linear program
by introducing additional slack variables (also denoted by $c_{i}$,
in an abuse of notation):
\begin{alignat}{2}
 & \underset{}{\optmin} & \quad & \sum_{i\in\mathcal{S}}c_{i}\qquad\text{(over \ensuremath{c_{i},s_{i},q_{ij},\alpha_{i},p_{i}})}\label{eq:opf_final}\\
 & \underset{}{\optst} & \quad & \text{\eqref{eq:opf-qij}--\eqref{eq:opf-di}}\nonumber \\
 &  &  & c_{i}\geq a_{ir}^{T}s_{i}+b_{ir},\quad i\in\mathcal{S},\ r=1,2,\dots,R_{i}.\nonumber 
\end{alignat}
It can be verified that problem~(\ref{eq:opf_final}) has the same
form as problem~(\ref{eq:prob_opt}), and hence its optimal value
is a piecewise concave function of $\theta$.

\section{Numerical examples\label{sec:Examples}}

This section demonstrates the numerical aspect of convex OUQ through
two examples. The first example compares bounds obtained by asymmetric
and incomplete information using convex OUQ and the algorithm introduced
by Bertsimas and Popescu~\cite{bertsimas2005optimal}. The second
example follows from the application in revenue maximization (with
stochastic supplies) given in section~\ref{sec:Applications}; in
particular, it presents a scenario where convex OUQ is (up to the
authors knowledge) the only applicable convex formulation, since previous
formulations do not handle arbitrary piecewise concave objective functions.

\paragraph{Bound on the tail of Gaussian distributions}

This example applies convex OUQ and the algorithm introduced by Bertsimas
and Popescu~\cite{bertsimas2005optimal} in order to compute an upper
bound of~$\mathbb{P}(\theta\geq a)$ (where~$a$ is a given constant)
in presence of incomplete and asymmetric information.

To apply convex OUQ, we assume that we are given the constraints 
\[
\mathbb{E}[\theta]=M_{1},\qquad\mathbb{E}[\theta^{2}]\leq M_{2},\qquad\mathbb{E}[|\theta|]\leq M_{1}^{+}.
\]
To apply Bertsimas and Popescu~\cite{bertsimas2005optimal} (which
has not been designed to incorporate the constraint~$\mathbb{E}[|\theta|]\leq M_{1}^{+}$),
we assume that we are given the constraints 
\[
\mathbb{E}[\theta^{q}]=M_{q},\quad q=1,2,\dots,Q
\]
for some fixed~$Q$. All the constants, including~$\{M_{q}\}_{q=1}^{Q}$
and~$M_{1}^{+}$, are computed from the standard Gaussian distribution~$\mathcal{N}(0,1)$.

Table~\ref{tab:optval_pwc} lists the results obtained from the two
algorithms. As a baseline of comparison, it also lists the exact value
of~$\mathbb{P}(\theta\geq x)$ as computed by numerically evaluating
the integral
\[
\frac{1}{\sqrt{2\pi}}\int_{x}^{\infty}\exp(-\theta^{2}/2)\, d\theta.
\]
It can be seen that the above integral is related to the complementary
error function $\mathrm{erfc}(\cdot)$ as follows:
\[
\frac{1}{\sqrt{2\pi}}\int_{x}^{\infty}\exp(-\theta^{2}/2)\, d\theta=\frac{1}{2}\mathrm{erfc}(x/\sqrt{2}),
\]
where the complementary error function $\mathrm{erfc}(\cdot)$ is
defined as: 
\[
\mathrm{erfc}(x)\triangleq\frac{2}{\sqrt{\pi}}\int_{x}^{\infty}\exp(-\theta^{2})\, d\theta.
\]
The algorithm by Bertsimas and Popescu eventually gives a tighter
bound by using more moment-based information, since it is capable
of incorporating equality constraints. On the other hand, when higher
moment ($Q\geq6$) information is unavailable, convex OUQ gives a
better bound by being able to handle more types of constraints such
as~$\mathbb{E}[|\theta|]\leq M_{1}^{+}$. In terms of computational
complexity, the algorithm by Bertsimas and Popescu requires solving
a semidefinite program, whereas our convex OUQ formulation only requires
solving a second-order cone program. It can be seen from Table~\ref{tab:optval_pwc}
that the computational time of our convex OUQ formulation is also
comparable to the algorithm by Bertsimas and Popescu.

\begin{table}[tbph]

\begin{center}
\begin{tabular}{|c|c|c|}
\hline 
Method & Upper bound for $\mathbb{P}(\theta\geq a)$ & CPU time (in seconds, 100 runs)\tabularnewline
\hline 
B \& P ($Q=2$) & $0.6400$ & $8.29$\tabularnewline
\hline 
B \& P ($Q=4$) & $0.6074$ & $11.05$\tabularnewline
\hline 
B \& P ($Q=6$) & $0.4964$ & $14.21$\tabularnewline
\hline 
convex OUQ ($Q=2$) & $0.5319$ & $11.50$\tabularnewline
\hline 
Exact value & $0.2266$ & ---\tabularnewline
\hline 
\end{tabular}
\end{center}
\caption{Upper bounds obtained by convex OUQ and the algorithm in Bertsimas
and Popescu (B \& P)~\cite{bertsimas2005optimal}. The exact value
is also listed for reference. In all results, the constant $a=0.75$.
The CPU time is measured on a laptop computer equipped with a dual-core
2.5 GHz Intel Core i5 processor and 4 GB of RAM. All the optimization
problems are solved in MATLAB (R2012b) using CVX (Version 2.1, Build
1079) with the Mosek solver (Version 7.0.0.106). \label{tab:optval_pwc}}
\end{table}

\paragraph{Revenue estimation}

This numerical example follows from problem (\ref{eq:prob-revenue})
presented in section~\ref{sec:Applications}. For simplicity, this
example considers 3 customers, so that the revenue~$f$ can be defined
as
\[
f(\theta)=\max_{k=1,2,3}f^{(k)}(\theta),
\]
where~$f^{(k)}$ is the revenue due to customer~$k$, and $\theta$
is the (random) quantity of the goods. We choose
\[
f^{(k)}(\theta)=\begin{cases}
a_{k}(\theta-b_{k})^{2}+c_{k}, & \quad\theta\leq b_{k}\\
c_{k}, & \quad\theta>b_{k},
\end{cases}
\]
where~$a_{k}<0$ models the rate at which the price from customer~$k$
drops, $b_{k}$ is the maximum quantity that customer~$k$ is willing
to purchase, and~$c_{k}$ is maximum revenue from customer~$k$.
It can be verified that each~$f^{(k)}$ is concave, and thus~$f$
is piecewise concave. The functions~$f$ and each~$f^{(k)}$ are
plotted in Figure~\ref{fig:pw_concave_revenue} based on the parameters
used in this example.

\begin{figure}
\centering{}\includegraphics[width=0.5\textwidth]{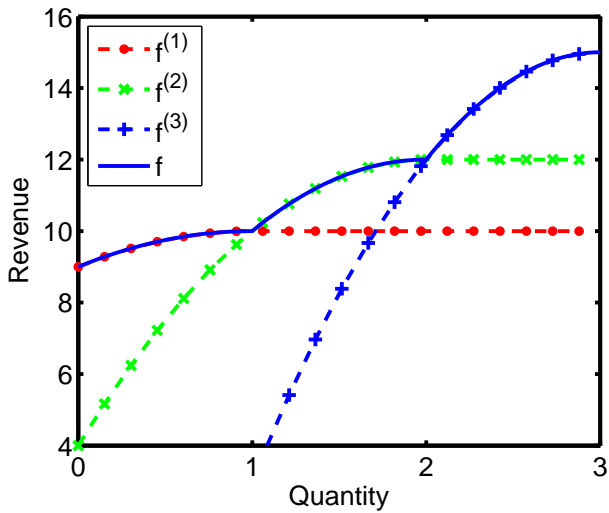}\caption{Solid: best revenue from all customers. Dashed: revenue from individual
customers. The parameters are: $a=(-1,-2,-3)$, $b=(1,2,3)$, and~$c=(10,12,15)$.}
\label{fig:pw_concave_revenue}
\end{figure}

The information constraints include constraints on the first and second
moments, i.e., 
\[
\mathbb{E}[\theta]=\mu,\qquad\mathbb{E}[\theta^{2}]\leq\mu^{2}+\sigma^{2},
\]
and tail probabilities
\[
\mathbb{P}(\theta\leq\theta_{L})\leq\delta_{L},\qquad\mathbb{P}(\theta\geq\theta_{H})\leq\delta_{H}.
\]
These information constraints specify the first and second order statistics
on the (random) quantity of the goods $\theta$, as well as the probabilities
that the quantity $\theta$ goes below a given lower bound $\theta_{L}$
or above a given upper bound $\theta_{H}$, respectively. As we mentioned
in Example~\ref{ex:pwcvx_tail}, both~$I(\theta\leq\theta_{L})$
and~$I(\theta\geq\theta_{H})$ are piecewise convex, and each can
be constructed from 2 convex functions. Therefore, we have~$K=3,$
$L_{1}=2$, and~$L_{2}=2$. Consequently, the total number of Dirac
masses is~$K\cdot L_{1}\cdot L_{2}=12$. 

The corresponding convex optimization problem that solves for an upper
bound of the expected revenue~$\mathbb{E}[f(\theta)]$ is
\begin{alignat*}{2}
 & \underset{\{p_{k,l_{1},l_{2}},\gamma_{k,l_{1},l_{2}}\}_{k,l_{1},l_{2}}}{\optmax} & \quad & \sum_{k,l_{1},l_{2}}p_{k,l_{1},l_{2}}f^{(k)}(\gamma_{k,l_{1},l_{2}}/p_{k,l_{1},l_{2}})\\
 & \underset{\phantom{\{p_{k,l_{1},l_{2}},\gamma_{k,l_{1},l_{2}}\}_{k,l_{1},l_{2}}}}{\optst} & \quad & \sum_{k,l_{1},l_{2}}p_{k,l_{1},l_{2}}=1,\qquad p_{k,l_{1},l_{2}}\geq0,\quad\forall k,l_{1},l_{2}\\
 &  &  & \sum_{k,l_{1},l_{2}}\gamma_{k,l_{1},l_{2}}=\mu\\
 &  &  & \sum_{k,l_{1},l_{2}}\gamma_{k,l_{1},l_{2}}^{2}/p_{k,l_{1},l_{2}}\leq\mu^{2}+\sigma^{2}\\
 &  &  & p_{k,1,l_{2}}\leq\delta_{L},\quad\gamma_{k,2,l_{2}}\geq p_{k,2,l_{2}}\theta_{L},\quad\forall k,l_{2}\\
 &  &  & p_{k,l_{1},1}\leq\delta_{H},\quad\gamma_{k,l_{1},2}\leq p_{k,l_{1},2}\theta_{H},\quad\forall k,l_{1}.
\end{alignat*}
Figure~\ref{fig:expected_revenue} shows the effect of changing the
second moment and the tail probabilities. As expected, loosening the
constraints (i.e., increasing either~$\sigma$ or $\delta_{L}$)
increases the upper bound of~$\mathbb{E}[f(\theta)]$. In the case
of changing the second moment (Figure~\ref{fig:expected_revenue_sigma}),
the upper bound stops increasing beyond a certain point, which implies
that the information constraint on the second moment is no longer
active.

\begin{figure}
\hspace*{\fill}
\subfigure[]{\centering\includegraphics[width=0.4\textwidth]{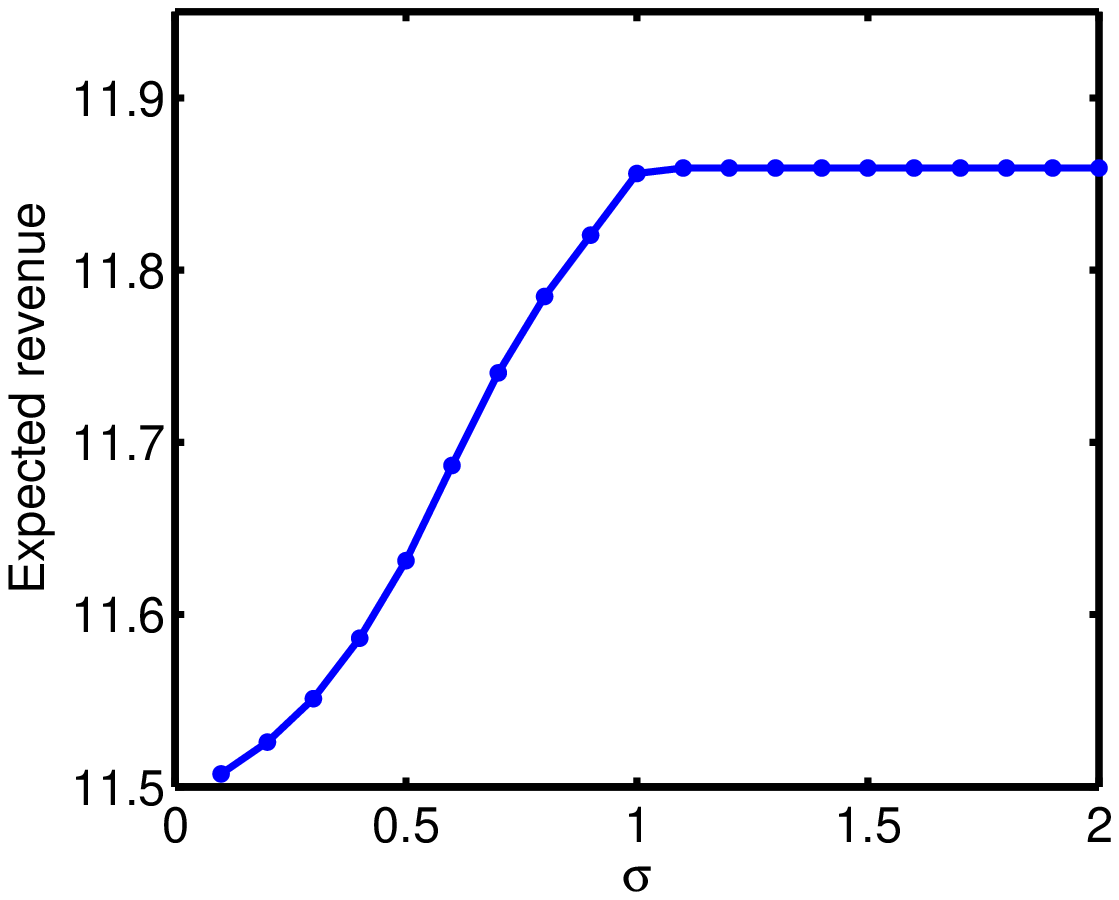}\label{fig:expected_revenue_sigma}}
\hspace*{\fill}
\subfigure[]{\centering\includegraphics[bb=0bp 0bp 331bp 219bp,width=0.4\textwidth]{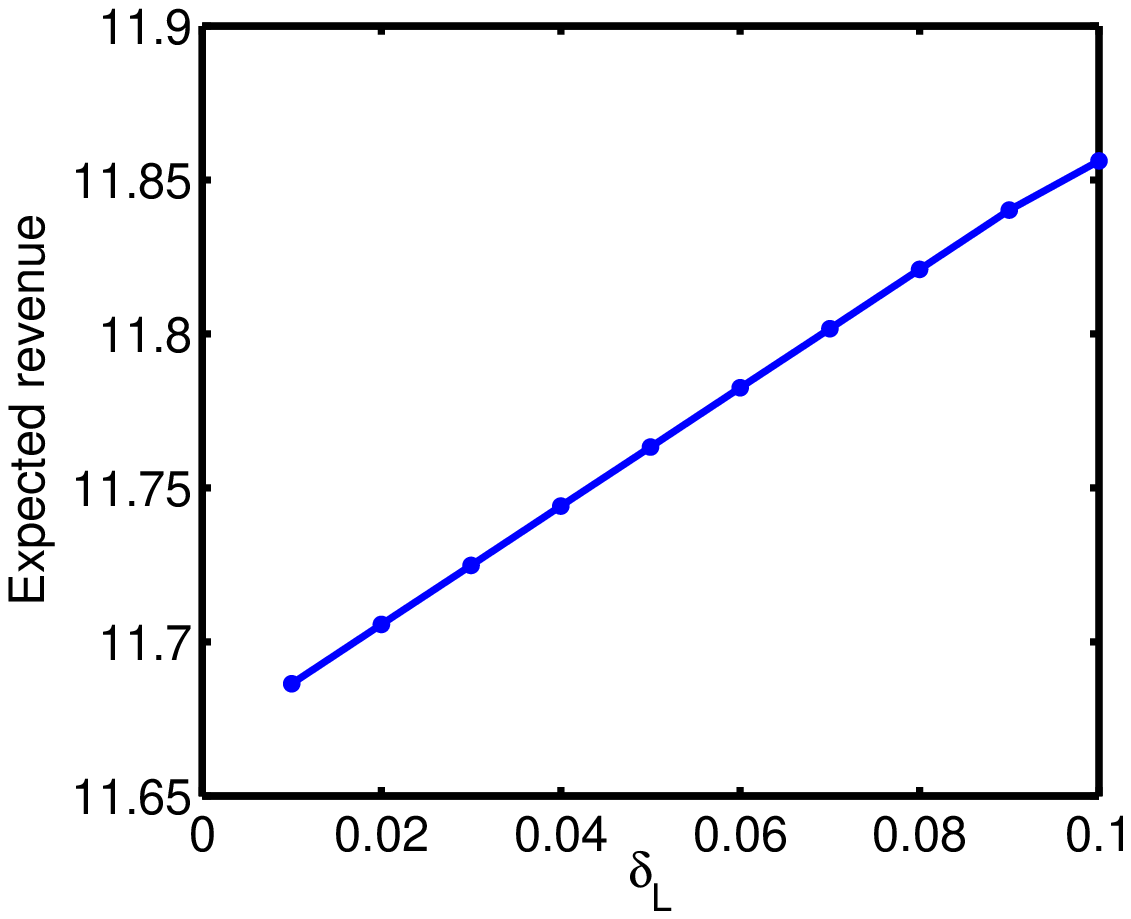}}
\hspace*{\fill}
\caption{Upper bound of the expected revenue computed by convex OUQ. (a) Dependence
on the standard deviation~$\sigma$; (b) Dependence on the (one-side)
tail probability~$\delta_{L}$.}
\label{fig:expected_revenue}
\end{figure}

\section{Conclusions}

This paper introduces the following new sufficient conditions under
which an OUQ problem can be reformulated as a convex optimization
problem: (1) the objective is piecewise concave, (2) the inequality
information constraints are piecewise convex, and (3) the equality
constraints are affine. Constraints on the support of the probability
distribution can also be incorporated, as long as the support is a
finite union of convex sets. We prove the result based on two different
approaches, which start with the primal and the dual forms of the
original OUQ problem, respectively. Through a number of examples,
we also illustrate the use of the convex OUQ formulation in several
different applications, such as estimating the maximum revenue of
selling goods to customers and the operational cost of power systems.

\section*{Acknowledgment}
This work was supported in part by NSF grant CNS-0931746 and AFOSR grant FA9550-12-1-0389.

\bibliographystyle{abbrv}
\bibliography{ref}

\end{document}